\documentclass[a4paper,10pt]{amsart}

\usepackage{amsmath}
\usepackage{amsthm}
\usepackage{epsfig}
\usepackage{amssymb}
\usepackage{amsmath}
\usepackage{amssymb}
\usepackage{amsfonts}
\usepackage{amsxtra}

\usepackage[utf8]{inputenc}
\usepackage{bm}
\usepackage{bbm}
\usepackage{esint}

\usepackage{tikz-cd}
\usetikzlibrary{decorations.pathreplacing}
\usepackage{euscript}
\usepackage{mathrsfs}
\usepackage{color}

\allowdisplaybreaks
\usepackage{tikz-cd}
\usetikzlibrary{decorations.pathreplacing}

\usepackage{enumitem}
\setenumerate{label={\rm (\alph{*})}}

\usepackage{xcolor}
\colorlet{ColorPink}{red!30}
\usepackage{graphicx}

\usepackage[colorlinks=true, linktocpage=true, linkcolor=red!70!black, citecolor=green!50!black]{hyperref}

\newcommand{\R}{\mathbb R}

\newcommand{\dif}{\operatorname{d}\!}
\newcommand{\lebe}{L}

\newcommand{\imag}{\operatorname{i}}
\newcommand{\locc}{\loc}
\newcommand{\hold}{C}
\newcommand{\bv}{\setBV}
\newcommand{\bd}{\ensuremath{{\mathrm{BD}}}} % bounded variation
\newcommand{\ball}{B}
\newcommand{\A}{\mathbb{A}}
\newcommand{\mres}{%
   \,\raisebox{-.127ex}{\reflectbox{\rotatebox[origin=br]{-90}{$\lnot$}}}\,%
 }

\renewcommand{\mres}{\mathbin{\vrule height 1.6ex depth 0pt width
0.13ex\vrule height 0.13ex depth 0pt width 1.3ex}}

\numberwithin{equation}{section}

\theoremstyle{theorem}

\newtheorem{theorem}{Theorem}[section]
\newtheorem{lemma}[theorem]{Lemma}
\newtheorem{proposition}[theorem]{Proposition}

\newtheorem{corollary}[theorem]{Corollary}

\newtheorem{remark}[theorem]{Remark}

\usepackage[comment,enumitem,hyperref,indicator,theorem_section,frak,mathbb,Rn]{paper_diening}

\providecommand{\opA}{\A}
\renewcommand{\dashint}{\fint}

\begin{document}

\renewcommand{\baselinestretch}{1.05}

\title[Continuity points via Riesz potentials]{Continuity points via Riesz potentials \\ for $\mathbb{C}$-elliptic operators}
\author[L.~Diening]{Lars Diening}
\address[L.~Diening]{Universit\"{a}t Bielefeld, Fakult\"{a}t für Mathematik, 33501 Bielefeld, Germany}
\author[F.~Gmeineder]{Franz Gmeineder}
\address[F.~Gmeineder]{Mathematical Institute, University of Bonn, Endenicher Allee 60, 53115 Bonn, Germany}
\thanks{\emph{Acknowledgment}: F.G. acknowledges financial support by the Hausdorff Centre for Mathematics, Bonn, and the University of Bielefeld.}

\maketitle
\begin{abstract}
  We establish a Riesz potential criterion for Lebesgue continuity
  points of functions of bounded $\A$-variation, where~$\A$ is a
  $\mathbb{C}$-elliptic differential operator of arbitrary order. This
  result might even be of interest for classical functions of bounded
  variation.
% In particular, the methods developed in this paper give a canonical approach to the control of higher order terms of the elements in the nullspace of $\A$.   
\end{abstract}

% \setcounter{tocdepth}{1}
%\tableofcontents
\section{Introduction}

%\subsection{$\lebe^{1}$-estimates}
Functions of bounded variation are a vastly studied subject, mainly as
they form the natural function space framework for a variety of
variational problems. Hence it is particularly important to understand
their fine properties, and a wealth of contributions on this theme is
available, cf. \textsc{Ambrosio} et al. \cite{AFP} and the references
therein. When dealing with \emph{full} gradients, powerful tools such
as the coarea formula are available, facilitating the proofs of key
results in this context such as the absolute continuity of $Du$ for
$\mathscr{H}^{n-1}$ or in the study of Lebesgue discontinuity points.

Various variational problems, however, require to work with more
general differential operators than the usual gradient, see
\cite{FuchsSeregin} for a comprehensive account of problems from
elasticity or plasticity. To provide a unifying approach to this
topic, let $\A$ be a $k$-th order, homogeneous, constant-coefficient
differential operator on $\R^{n}$ between the two finite dimensional
real vector spaces $V$ and $W$. By this we understand that $\A$ has a
representation
\begin{align}\label{eq:form}
  \A u = \sum_{\substack{\alpha\in\mathbb{N}_{0}^{n} \\ |\alpha|=k}}\A_{\alpha}\partial^{\alpha}u, 
\end{align}
where $\A_{\alpha}\in\mathscr{L}(V;W)$ are fixed linear maps; note
that $\partial^{\alpha}$ acts compontentwisely on a function
$u\,:\,\R^{n}\to V$. If $\A$ is \emph{elliptic}
(cf. \textsc{H\"{o}rmander} \cite{Hoermander} or \textsc{Spencer}
\cite{Spencer}), meaning that the Fourier symbol
\begin{align}
  \A[\xi] = \sum_{|\alpha|=k}\xi_{\alpha}\A_{\alpha}\colon V\to W
\end{align}
is injective for all $\xi\in\R^{n}\setminus\{0\}$, then elementary
Fourier multiplier techniques establish that for each $1<p<\infty$
there exists $c_{p}>0$ such that there holds
\begin{align}\label{eq:CZ}
  \|D^{k}u\|_{\lebe^{p}(\R^{n})}\leq c_{p}\|\A u\|_{\lebe^{p}(\R^{n})}\qquad\text{for all}\;u\in\hold_{c}^{\infty}(\R^{n};V). 
\end{align}
Inequalities of this type are usually referred to as
\textsc{Korn}-type inequalities, cf. \cite{Korn,Mosolov,BD} for
instance. By a foundational result of \textsc{Ornstein}
\cite{Ornstein}, estimate \eqref{eq:CZ} does not persist for $p=1$ in
general. Indeed, by the sharp version as recently established by
\textsc{Kirchheim \& Kristensen} \cite{KK0,KK}, validity of
\eqref{eq:CZ} for $p=1$ is equivalent to the existence of some
$T\in\mathscr{L}(W;V\odot^{k}\R^{n})$ such that $D^{k}=T\circ\A$. In
this case, however, \eqref{eq:CZ} trivalises, leading to the informal
metaprinciple that \emph{there are no non-trivial
  $\lebe^{1}$-estimates}.

%\subsection{Functions of bounded $\A$-variation}
Let $u\in\lebe^{1}(\R^{n};V)$. Given a differential operator of the form \eqref{eq:form}, we say that $u$ is of \emph{bounded $\A$-variation} and write $u\in\bv^{\A}(\R^{n})$ if and only if the distributional differential expression $\A u$ can be represented by a finite $W$-valued Radon measure, denoted $\A u\in\mathscr{M}(\R^{n};W)$. The space $\bv_{\locc}^{\A}$ then is defined in the obvious way. This class of function spaces has been introduced in \cite{BDG,GR1}, and by \textsc{Ornstein}'s Non-Inequality, we have $\bv\subsetneq \bv^{\A}$ in general. As a major obstruction in the study of $\bv^{\A}$-maps, the failure of \eqref{eq:CZ} for $p=1$ equally rules out the use of full gradient techniques. Still, as is by now well-known, several properties of $\bv$-maps can be proven to hold for $\bv^{\A}$-maps, too, and we refer the reader to \cite{VS13,VS14,VS14a,BDG,GR1,GR2,GR3,Raita1,Raita2} for a comprehensive list of results in this area. 

Based on \textsc{Smith} \cite{Smith}, in \cite{BDG} \textsc{Breit} and
the authors isolated a key property of first order differential
operators $\A$ to yield \emph{boundary trace embeddings}
$\bv^{\A}(\Omega)\hookrightarrow\lebe^{1}(\partial\Omega;V)$, namely
\emph{$\mathbb{C}$-ellipticity}. We say that a differential operator
of the form \eqref{eq:form} is $\mathbb{C}$-elliptic provided the
complexified Fourier symbol
\begin{align*}
\A[\xi]\colon V+\imag V \to W+\imag W\qquad\text{is injective for all}\;\xi\in\mathbb{C}^{n}\setminus\{0\}. 
\end{align*} 
As a particular consequence of \cite[Thm.~1.1]{BDG}, $\mathbb{C}$-elliptic differential operators have finite dimensional nullspace (in $\mathscr{D}'(\R^{n};V)$) consisting of polynomials of a fixed maximal degree exclusively. As is well-known from the classical $\bv$-theory, interior traces are instrumental for a description of the jump parts and hence the set of Lebesgue discontinuity points. 
This is the starting point for the present paper, where we aim to introduce Riesz potential techniques in the study of $\bv^{\A}$-functions for the particular case of $\mathbb{C}$-elliptic differential operators. 

It is easy to see that finiteness of the fractional maximal operator $\mathcal{M}_{k}(\A u)$ with 
\begin{align*}
\mathcal{M}_{k}\mu(x):=\sup_{\ball\ni x}\frac{|\mu|(\ball(x,r))}{r^{n-k}},\qquad \mu\in\mathcal{M}(\R^{n};W),\;x\in\R^{n},
\end{align*}
cannot yield a criterion for Lebesgue continuity points. In fact,
consider $u=\mathbbm{1}_{\{x: |x|<1, x_{n}>0\}}$ and $k=1$, for which
$\mathcal{M}_{1}(Du)(x)<\infty$ for any
$x\in\partial\ball(0,1)$. Opposed to this, the main result of the
present paper is that the Riesz potentials are sufficiently powerful
to \emph{detect Lebesgue continuity points} indeed. Given a Radon
measure $\mu\in\mathscr{M}(\R^{n};W)$ and $s>0$, we define the
\emph{Riesz potential of $\mu$ of order $s$} by
\begin{align}\label{eq:defRiesz}
  \mathcal{I}_{s}(\mu)(x_{0}):= 
  \int_{\ball(x_{0},r)}\frac{\dif|\mu|(y)}{|x_{0}-y|^{n-s}}<\infty,\qquad x_{0}\in\R^{n}. 
\end{align}
Let $S_{u}$ denote the set of \emph{non-Lebesgue points} of a map
$u\in\lebe_{\locc}^{1}(\R^{n};V)$.
\begin{theorem}
  \label{thm:main}
  Let $n\geq 2$, $k\geq 1$ and let $\A$ be a $k$-th order $\mathbb{C}$-elliptic
  differential operator of the form \eqref{eq:form}. Then the
  following hold:
  \begin{enumerate}
  \item\label{item:main1} If $u\in\bv_{\locc}^{\A}(\R^{n})$ and
    $x_{0}\in\R^{n}$ are such that for some $r>0$ there holds
    \begin{align}\label{eq:maineq}
      \mathcal{I}_{k}(\A u\mres\ball(x_{0},r))(x_{0})<\infty, 
    \end{align}
    then $x_{0}$ is a Lebesgue point for $u$, i.e. $x_{0}\in S_{u}^{\complement}$. 
  \item\label{item:main2} If $k=n$, then any
    $u\in\bv_{\locc}^{\A}(\R^{n})$ has a continuous representative. If
    $x \in \Rn$ and $r>0$, then for every~$y \in B(x,r/2)$ the
    continuous representative satisfies
    \begin{align}
      \label{eq:cont-est-thm}
      \abs{u(x)-u(y)}
      &\leq  c\,\abs{\bbA u}(B \setminus \set{x}) 
        + c\,\frac{\abs{x-y}}{r} \dashint_{B} \abs{u-\mean{u}_{B}}\dif z,
    \end{align}
    where $B := B(x,r)$.  
  \item\label{item:main3} If $k>n$, then any
    $u\in\bv_{\locc}^{\A}(\R^{n})$ has a $C^{n-k}$
    representative.
    If
    $x \in \Rn$ and $r>0$, then for every~$y \in B(x,r/2)$ the
    continuous representative satisfies
    \begin{align}
      \label{eq:cont-est-thm}
      \begin{split}
      \abs{(\nabla^{k-n} u)(x) -(\nabla ^{k-n} u)(y)} 
      &\leq  c\,\abs{\bbA u}(B \setminus \set{x})
      \\
      &\quad
        + c\,\frac{\abs{x-y}}{r} \dashint_{B} \abs{\nabla^{k-n}
        u-\mean{\nabla^{k-n} u}_{B}}\dif z, 
        \end{split}
    \end{align}
    where $B := B(x,r)$.  
  \end{enumerate}
\end{theorem}
Throughout the paper we consider the case~$n\geq 2$. Indeed, for $n=1$
the only elliptic operators are of the form $a\frac{\dif}{\dif x}$
for $a\in\R\setminus\{0\}$. Hence, if $n=1$, $\bv^{\A}(\R)=\bv(\R)$
for all elliptic operators $\A$ and so \ref{item:main2} of
Theorem~\ref{thm:main} fails. The case $n \geq 2$ differs from the
case~$n=1$ in the fact that annuli are connected for~$n\geq 2$ and
allow therefore to deduce a Poincar\'{e}-type inequality on annuli,
cf.~Corollary \ref{cor:poincare-annulus}. So from now on let~$n \geq 2$.

Let us define, for
$u\in\bv^{\A}(\R^{n})$,
\begin{align}
  \Sigma_{u}:=\left\{x_{0}\in\R^{n}\colon\;\mathcal{I}_{k} (\A u\mres\ball(x_{0},r))=\infty\;\text{for all}\;r>0\right\}.  
\end{align}
By Theorem~\ref{thm:main}~\ref{item:main1},
$\Sigma_{u}^{\complement}\subset S_{u}^{\complement}$ and so
$S_{u}\subset \Sigma_{u}$. From general measure and potential
theoretic principles it then follows that the Hausdorff dimension of
$\Sigma_{u}$ and hence $S_{u}$ cannot exceed $(n-k)$ for a
$\mathbb{C}$-elliptic operator $\A$ and $u\in\bv^{\A}(\R^{n})$. If
$k\geq n \geq 2$, then by Theorem~\ref{thm:main}\ref{item:main2} and
Remark~\ref{rem:Linfty} more can be said: In this case, we have
$\bv^{\A}(\Omega)\hookrightarrow C^{k-n}(\overline{\Omega};V)$ for
every bounded Lipschitz domain. Thus, we obtain
an independent proof \emph{for the case of $\mathbb{C}$-elliptic
  operators} of a more general borderline embedding theory developed
by \textsc{Raita} \& \textsc{Skorobogatova} \cite[Thm.~1.1, 1.3]{Raita2}
for elliptic and cancelling operators. The case $k>n$ establishes how \cite[Thm.~1.4]{Raita1} can be strengthened in the case of $\mathbb{C}$-elliptic operators and seems to be new. 

% \subsection{Structure of the paper}
Let us finally explain the structure of the paper. In
Section~\ref{sec:prelims} we gather notation and background facts on
$\bv^{\A}$-functions and \Poincare{} type inequalities. In
Section~\ref{sec:oscill-estim}, we provide oscillation estimates for
functions $u$ in terms of the Riesz potentials of $\A u$, which will
eventually yield the proof of Theorem~\ref{thm:main} in
Section~\ref{sec:proofofthm}.

\section{Preliminaries}
\label{sec:prelims}

We start by fixing notation. Throughout, $B=B(x_0,r)$ denotes the open
ball in~$\setR^n$ of radius $r>0$ centered at $x_{0}$. We also often
use annuli of the form~$\frA = B \setminus \lambda \overline{B}$ for
some~$\lambda \in [0,\frac 12]$. For $s>0 $ we denote by $sB$,
resp. $s\frA$, the balls, resp. annuli, that are scaled by the
factor~$s$ by keeping the center in place.  As usual, we denote
$\mathscr{L}^{n}$ or $\mathscr{H}^{n-1}$ the $n$-dimensional Lebesgue
or $(n-1)$-dimensional Hausdorff measures, respectively, and sometimes
abbreviate $|U|:=\mathscr{L}^{n}(U)$. Whenever $B$ is an open ball, we
use the equivalent notations
\begin{align*}
  \mean{u}_{\ball}:=\dashint_{\ball}u\dif x :=\frac{1}{|\ball|}\int_{\ball}u\dif x. 
\end{align*}
Given a finite dimensional real vector space $X$, the finite,
$X$-valued Radon measures on the open set $\Omega$ are denoted
$\mathscr{M}(\Omega;X)$. Also, given $\mu\in\mathscr{M}(\Omega;X)$ and
Borel subsets $A,U$ of $\Omega$, we define
$(\mu\mres A)(U):=\mu(A\cap U)$. Finally, by $c>0$ we denote a generic
constant that might change from line to line and shall only be
specified if its precise dependence on other parameters is required.

For the following, let $\A$ be a $\mathbb{C}$-elliptic differential
operator of the form \eqref{eq:form}. We then record the following
facts, retrievable from \cite{BDG,GR1}. $\mathbb{C}$-ellipticity of
$\A$ implies that for any connected, open subset $\Omega\subset\R^{n}$
the nullspace
$N(\A;\Omega):=\{u\in\mathscr{D}'(\Omega;V)\colon\;\A u=0\}$ of $\A$
is finite dimensional and is a subset of the set of polynomials
$\mathscr{P}_{m}(\Omega;V)$ of a fixed maximal degree
$m\in\mathbb{N}$. As such, for any open, bounded and connected
$\Omega\subset\R^{n}$, $N(\A;\Omega)\subset\lebe^{2}(\Omega;V)$ and we may
define $\Pi_{\Omega}$ to be the $\lebe^{2}$-orthogonal projection onto
$N(\A;\Omega)$.

Following \textsc{Smith} \cite{Smith} or
\textsc{Ka\l{}amajska}~\cite{Kalamajska} we can represent
every~$u \in \bv^\A(\Omega)$, where~$\Omega$ is a star-shaped domain
with respect to a ball~$B_\Omega$ by a projection and a convolution
of~$\A u$ with a Riesz potential kernel. In particular, there
exists~$m \in \setN_0$ depending on~$\opA$ and an integral kernel
$\mathfrak{K}_\Omega\colon \Omega \times \Omega \to\mathscr{L}(W;V)$
which is $C^\infty$ off the diagonal $\set{x=y}$ and satisfies
$|\partial_x^\alpha \partial_y^\alpha \mathfrak{K}_{\Omega}(x,y)|\leq
c_{\alpha,\beta} \,|x-y|^{k-n-\abs{\alpha}-\abs{\beta}}$ such that
\begin{align}\label{eq:repformula}
  u(x) = \mathbb{P}_{B_\Omega}^mu(x) +
  \int\limits_{\convexhull(\overline{B_\Omega} \cup \set{x})}
  \mathfrak{K}_\Omega(x-y)\A u(y)\dif y
\end{align}
for $\mathscr{L}^{n}$-a.e. $x\in\Omega$, where $\bbP^m_{B_\Omega} u$ denote the averaged Taylor polynomial of
order~$m$ with respect to the ball~$B_\Omega$ and
$\convexhull(\overline{B_\Omega} \cup \set{x})$ is the closed convex hull of
$B_\Omega \cup \set{x}$.

It follows from this
representation and the property that
$\nabla^\ell \mathbb{P}_{B_\Omega}^m = \mathbb{P}_{B_\Omega}^{m-\ell}
\nabla^\ell$ that for every $l=0,\dots, k-1$ there exists 
$\mathfrak{K}_\Omega^\ell \colon \Omega \times \Omega
\to\mathscr{L}(W;V)$ which is $C^\infty$ off the diagonal $\set{x=y}$
and satisfies
$|\partial_x^\alpha \partial_y^\alpha \mathfrak{K}_{\Omega}^{\ell}(x,y)|\leq
c_{\alpha,\beta} \,|x-y|^{k-\ell-n-\abs{\alpha}-\abs{\beta}}$ such
that
\begin{align}\label{eq:repformula_high}
  \nabla^\ell u(x) =  \mathbb{P}_{B_\Omega}^{n-\ell}
  \nabla^\ell u(x) +
  \int\limits_{\convexhull(\overline{B_\Omega} \cup \set{x})}
  \mathfrak{K}_\Omega^\ell(x-y)\A u(y)\dif y.
\end{align}
Working from this representation, we infer the usual \Poincare{}-type estimates:
\begin{lemma}[\Poincare{} for star-shaped domains]
  \label{lem:poincare-ball-pre}
  Let $\Omega$ be star-shaped domain with respect to the
  ball~$B_\Omega$ with radius~$r_B$ and
  $\diameter(\Omega) \leq c\, r_B$.  Let $\A$ be a $k$-th order
  $\mathbb{C}$-elliptic differential operator of the form
  \eqref{eq:form}. Then there exists a constant $c=c(\A)>0$ such that
  \begin{align}
    \label{eq:poincare-ball-pre}
    \sum_{\ell=0}^{k-1} \dashint_{\Omega} r_B^\ell
    \abs{\nabla^\ell u-\nabla^\ell \mathbb{P}^m_{B_\Omega}u}\dif x
    &\leq c\,r^{k}\frac{|\A u|(\Omega)}{\abs{\Omega}}.
  \end{align}
  holds for all $u\in\bv^{\A}(\Omega)$. Recall that 
  $\nabla^\ell \mathbb{P}_{B_\Omega}^m =
  \mathbb{P}_{B_\Omega}^{m-\ell} \nabla^\ell$.
\end{lemma}
In the following we show how to replace the averaged Taylor polynomial
in this formula by the projection~$\Pi_B$.
\begin{corollary}
  \label{cor:poincare-ball}
  Under the assumptions of Lemma~\ref{lem:poincare-ball-pre} there holds
  \begin{align}
    \label{eq:poincare-ball}
    \sum_{\ell=0}^{k-1} \dashint_{\Omega} r_B^\ell
    \abs{\nabla^\ell u- \nabla^\ell \Pi_Bu}\dif x
    &\leq c\,r^{k}\frac{|\A u|(\Omega)}{\abs{\Omega}}.
  \end{align}
\end{corollary}
\begin{proof}
  It remains to show
  \begin{align*}
    \textrm{I} := \sum_{\ell=0}^{k-1} \dashint_\Omega r^\ell_\Omega \abs{
    \nabla^\ell (\mathbb{P}^{m}_{B_\Omega} u - \Pi_B
    u)}\,dx 
    &\leq c\,r^{k}\frac{|\A u|(\Omega)}{\abs{\Omega}}.
  \end{align*}
  Since $N(\A) \subset \mathscr{P}_m$ we have $\mathbb{P}^{m}_{B_\Omega} u - \Pi_B
    u = \mathbb{P}^{m}_{B_\Omega} u - \Pi_B
    \mathbb{P}^m_{B_\Omega} u$. Moreover, 
  for any~$p \in \mathscr{P}_m$ there holds
  \begin{align*}
    \sum_{\ell=0}^{k-1} \dashint_\Omega r^\ell_\Omega \abs{
    \nabla^\ell p - \nabla^\ell p}\,dx  &\leq c\, \dashint_\Omega \abs{\opA p}\,dx,
  \end{align*}
  since both sides define a norm on the finite dimensional
  space~$\mathscr{P}_l / N(\opA)$ and vanish on~$N(\opA)$. Thus,
  \begin{align*}
    \textrm{I} &\leq \sum_{\ell=0}^{k-1} \dashint_\Omega r^\ell_\Omega \abs{
                 \nabla^\ell (\mathbb{P}^{m}_{B_\Omega} u - \Pi_B
                 \mathbb{P}^m_{B_\Omega} u)}\,dx
    \\
               &\leq c\, \dashint_B \abs{\A (\mathbb{P}^{m}_{B_\Omega} u - \Pi_B
                 \mathbb{P}^m_{B_\Omega} u)}\,dx
    \\
               &= c\, \dashint_\Omega \abs{\mathbb{P}^{m-k}_{B_\Omega}
                 \opA u}\,dx.
    \\
               &\leq c\, \dashint_\Omega \abs{
                 \opA u}\,dx
  \end{align*}
  using that
  $\nabla^\ell \mathbb{P}_{B_\Omega}^m =
  \mathbb{P}_{B_\Omega}^{m-\ell} \nabla^\ell$ and the
  $\lebe^1$-stability of $\mathbb{P}^{m-k}_{B_\Omega}$. The proof of the corollary is complete. 
\end{proof}
We also need the \Poincare{}-type inequality for annuli and punctured balls.
\begin{corollary}[\Poincare{} annuli]
  \label{cor:poincare-annulus}
  Let $n \geq 2$.  Let $B$ be ball with radius~$r_B$. Let $\A$ be a
  $k$-th order $\mathbb{C}$-elliptic differential operator of the form
  \eqref{eq:form}. Then there exists a constant $c=c(\A)>0$ such that
  the following holds. Let $\Omega$ be the
  annulus~$\frA_\lambda := B \setminus \lambda \overline{B}$ with
  $\lambda \in [0,\frac 12]$. Then
  \begin{align}
    \label{eq:poincare-annulus}
    \sum_{\ell=0}^{k-1} \dashint_{\Omega} r_B^\ell
    \abs{\nabla^\ell u-  \nabla^\ell \Pi_\Omega u}\dif x
    &\leq c\,r^{k}\frac{|\A u|(\Omega)}{\abs{\Omega}}.
  \end{align}
  holds for all $u\in\bv^{\A}(\Omega)$.
\end{corollary}
\begin{proof}
  The estimate as in Lemma~\ref{lem:poincare-ball-pre}
  involving~$\nabla^\ell \mathbb{P}^{m}_{B_\Omega} u$
  (where~$B_\Omega$ is a suitable sub-ball of~$\Omega$) follows
  for~$\Omega$ in fact by a standard argument and works in fact for
  any bounded Lipschitz domain. Since~$\Omega$ can be written as the
  finite union of overlapping subdomains~$\Omega_1, \dots, \Omega_N$
  (with $N$ depending only on~$n$) which are star shaped with respect
  to a ball. These subdomains can be constructed, such that
  $\Omega_j \cap \Omega_{j+1}$ contain a ball~$B_j$ of size equivalent
  to~$\Omega$ and $\Omega_j$ and $\Omega_{j+1}$ both are star shaped
  with respect to this ball~$B_j$. The difference of the averaged
  Taylor polynomials on two consecutive balls~$B_j$ and~$B_{j+1}$ can
  be estimated again by Lemma~\ref{lem:poincare-ball-pre}. Now, we can
  change $\mathbb{P}^m_{B_\Omega}$ to~$\Pi_\Omega$ exactly as in
  Corollary~\ref{cor:poincare-ball}.
\end{proof}
\begin{remark}
  \label{rem:john}
  Based on the technique introduced in \cite{DRS}, the
  Poincar\'{e}-type inequalities of the above form can moreover be
  established for bounded John domains, a fact that we shall pursue
  elsewhere.

  Note that, Corollary~\ref{cor:poincare-annulus} fails for~$n=1$. The
  problem is that the annuli are not connected for~$n=1$.
\end{remark}
It is well known that there exists a constant $c=c(n,m,\dim(V))>0$
such that for all
$q\in\mathscr{P}_{m}(\R^{n};V)$ and all balls~$B$ there holds
\begin{align}\label{eq:inverse}
  \frac{1}{c}\dashint_{B}|q|\dif x \leq \|q\|_{\lebe^{\infty}(B)}\leq c\dashint_{B}|q|\dif x. 
\end{align}
Such estimate are usually called \emph{inverse estimates}. In fact the
estimate follows by the equivalence of all norms on finite dimensional
spaces and scaling.

These inverse estimate on~$N(\A) \subset \mathscr{P}_m$ (for
bounded~$\Omega$) and the self-adjointness of~$\Pi_{\Omega}$ allows in
a standard way to extend~$\Pi_{\Omega}$ to a $L^1(\Omega)$ with
\begin{align}\label{eq:L1stability} 
  \dashint_{\Omega}|\Pi_{\Omega}u|\dif x \leq c \dashint_{\Omega}|u|\dif x. 
\end{align}
We later refer to this as the~$L^1$-stability of~$\Pi_\Omega$.

We need the following inverse estimates for polynomials that vanish at
the center of the ball.
\begin{lemma}
  \label{lem:inv-x}
  Let $m \in \setN_0$. Then there exists~$c=c(m,n)$ such that for all
  balls~$B$ with center~$x_0$,
  all~$\lambda \in (0,1]$ and all~$q \in \mathscr{P}_{m}(\R^{n};V)$ with
  $q(x_0)=0$, we have
  \begin{align*}
    \dashint_{\lambda B} \abs{q(x)}\dif x &\leq c\, \lambda
                                            \dashint_{B} \abs{q(x)}\,\dif x.
  \end{align*}
\end{lemma}
\begin{proof}
  By translation and scaling it suffices to establish the claim in the
  case~$B=B(0,1)$.  Because of $q(0)=0$ we may write
  $q(x) = x\, \cdot q_1(x)$. Thus
  \begin{align*}
    \dashint_{\lambda B(0,1)} \abs{q(x)}\dif x
    &=
      \dashint_{\lambda B(0,1)} \abs{x \cdot q_1(x)}\dif x
      \leq \lambda \max_{B(0,1)} \abs{q_1(x)}.
  \end{align*}
  Now,
  \begin{align*}
    \max_{x\in B(0,1)} \abs{q_1(x)}
    &\eqsim
      \dashint_{B(0,1)} \abs{x} \abs{q_1(x)} \dif x,
  \end{align*}
  since both terms are norms on the finite dimensional
  space~$\mathscr{P}_{m}(\R^{n};V)$. Hence,
  \begin{align*}
    \dashint_{\lambda B(0,1)} \abs{q(x)}\dif x
    &\leq c\,
      \lambda     \dashint_{B(0,1)} \abs{x} \abs{q_1(x)} \dif x
      = c\, \lambda \dashint_B \abs{q}\dif x,
  \end{align*}
  where we have used in the last step inverse estimates for
  polynomials.
\end{proof}

\section{Oscillation estimates and the proof of Theorem~\ref{thm:main}}
\label{sec:oscill-estim}
\subsection{Oscillation estimates}
Throughout the entire section, let $n\geq 2$. We moreover tacitly
assume $\A$ to be a $k$-th order $\mathbb{C}$-elliptic operator of the
form \eqref{eq:form} and let $u\in\bv_{\locc}^{\A}(\R^{n})$. Moreover,
we fix $x_{0}\in\R^{n}$, $r>0$ and put $\ball:=\ball(x_{0},r)$. Toward
Theorem~\ref{thm:main}, we begin with the following lemma which allows to
control oscillations of~$u - \mean{u}_B$ be means of $u - \Pi_B u$.
\begin{proposition}
  \label{pro:osc-sum}
  There exists $c=c(\A)>0$ such that for all $u\in\lebe^{1}(\ball;V)$
  there holds for each ball~$B$ and annulus~$\frA = B \setminus \frac
  14 \overline{B}$
  \begin{align}
    \label{eq:osc-sum-main}
    \begin{aligned}
      \dashint_{2^{-j}B} \abs{u - \mean{u}_{2^{-j}B}}\dif x &\leq
      2^{-j} c\,\dashint_{\frA} \abs{u - \mean{u}_{\frA}}\dif x
      \\
      &\quad + c \dashint_{2^{-j}B} \abs{u -\Pi_{2^{-j}B} u}\dif x
      \\
      &\quad + c \sum_{m=0}^{j} 2^{m-j}\dashint_{2^{-m}\frA} \abs{u
        -\Pi_{2^{-m}\frA} u}\dif x.
    \end{aligned}
  \end{align}
  Moreover,
  \begin{align}
    \label{eq:osc-sum-main-ann}
    \begin{aligned}
      \dashint_{2^{-j}\frA} \abs{u - \mean{u}_{2^{-j}\frA}}\dif x &\leq
      2^{-j} c\,\dashint_{\frA} \abs{u - \mean{u}_{\frA}}\dif x
      \\
      &\quad + c \sum_{m=0}^{j} 2^{m-j}\dashint_{2^{-m}\frA} \abs{u
        -\Pi_{2^{-m}\frA} u}\dif x.
    \end{aligned}
  \end{align}
\end{proposition}
\begin{proof}
  We begin with the proof of~\eqref{eq:osc-sum-main}.  By routine means,
  we then find
  \begin{align}
    \label{eq:osc-sum-1}
    \begin{aligned}
      \dashint_{2^{-j}B} \abs{u - \mean{u}_{2^{-j}B}}\dif x &\leq
      2\,\dashint_{2^{-j}B} \abs{u(x) - (\Pi_{2^{-j}B} u)(x_0)}\dif x
      \\
      & \leq 2\dashint_{2^{-j}B} \abs{u - \Pi_{2^{-j}B} u}\dif x
      \\
      &\quad +2\,\norm{\Pi_{2^{-j}B}u - \Pi_{2^{-j}\frA}
        u}_{L^\infty(2^{-j} B)}
      \\
      &\quad +2\dashint_{2^{-j}B} \abs{\Pi_{2^{-j}\frA} u - (\Pi_{2^{-j}\frA}
        u)(x_0)}\dif x
      \\
      &=: \mathrm{I}+\mathrm{II} +\mathrm{III}.
    \end{aligned}
  \end{align}
  The term~$\textrm{I}$ is already suitable for later. We estimate
  \begin{align*}
    \textrm{II} &= 2\,\norm{\Pi_{2^{-j}B} - \Pi_{2^{-j}\frA}
                  u}_{L^\infty(2^{-j} B)}
    \\
    &=
                  2\,\norm{\Pi_{2^{-j}\frA}(u - \Pi_{2^{-j}B}  u)}_{L^\infty(2^{-j} B)}
    \\
    &\leq c\,\norm{\Pi_{2^{-j}\frA}(u - \Pi_{2^{-j}B}  u)}_{L^\infty(2^{-j} \frA)}
    \\
    &\leq c\, \dashint_{2^{-j} \frA} \abs{u - \Pi_{2^{-j}B}  u}
      \dif x
    \\
    &\leq c\, \dashint_{2^{-j} B} \abs{u - \Pi_{2^{-j}B}  u}
      \dif x
  \end{align*}
  using inverse estimates in the penultimate step.
  Let us estimate~\textrm{III}.  For notational brevity, put
  $p_{j}:=\Pi_{2^{-j}\frA}u$.  Next we employ a telescope sum argument to
  bound the term $\mathrm{III}$ by
  \begin{align}
    \label{eq:osc-sum-2}
    \begin{aligned}
      \mathrm{III} = \dashint_{2^{-j}B}&|p_j(x) - p_j(x_0)| \dif x
      \leq\dashint_{2^{-j}B} \abs{p_0(x) - p_0(x_0)}\dif x
      \\
      & + \sum_{m=0}^{j-1} \dashint_{2^{-j}B} \abs{(p_{m+1} -p_m)(x) -
        (p_{m+1} -p_m)(x_0)}\dif x.
    \end{aligned}
  \end{align}
  In conclusion, by Lemmas~\ref{lem:inv-x} and inverse estimates
  \begin{align*}
    \lefteqn{\textrm{III}_m := \dashint_{2^{-j}B} \abs{(p_{m+1} -p_m)(x) - (p_{m+1}
    -p_m)(x_0)}\dif x} \qquad &
    \\
                              &\leq c\, 2^{m-j} \dashint_{2^{-m}B} \abs{(p_{m+1} -p_m)(x) -
                                (p_{m+1} -p_m)(x_0)}\dif x
    \\
                              &\leq c\, 2^{m-j} \bigg( \dashint_{2^{-m}B} \abs{p_{m+1}
                                -p_m}\dif x + \abs{(p_{m+1} -p_m)(x_0)}\bigg)
    \\
                              &\leq c\, 2^{m-j} \norm{p_{m+1} -p_m}_{L^\infty(2^{-m}B)}
    \\
                              &\leq c\, 2^{m-j} \dashint_{2^{-m}\frA \cap
                                2^{-m-1} \frA} \abs{p_{m+1}
                                -p_m} \dif x.
  \end{align*}
  Let us abbreviate $\frA_{m+1/2} := 2^{-m} \frA \cap 2^{-m-1}\frA$
  and $p_{m+1/2} := \Pi_{2^{-m} \frA \cap 2^{-m-1} \frA}$. We estimate
  \begin{align*}
    \lefteqn{\dashint_{\frA_{m+1/2}} \abs{p_{m+1}
    -p_m} \dif x} \qquad&
    \\
                        &\leq\dashint_{\frA_{m+1/2}} \abs{p_{m+1/2}
                          -p_{m+1}} \dif x + \dashint_{\frA_{m+1/2}} \abs{p_{m+1/2}
                          -p_m} \dif x
    \\
                        &= c\,  \dashint\limits_{\frA_{m+1/2}} 
                          \abs{\Pi_{\frA_{m+1/2}} u-\Pi_{2^{-m-1}\frA}u
                          } \dif x
                          + c\,  \dashint\limits_{\frA_{m+1/2}} \!\!\!
                          \abs{\Pi_{\frA_{m+1/2}} u-\Pi_{2^{-m}\frA}u
                          } \dif x
    \\
                        & \leq c\,  \dashint\limits_{\frA_{m+1/2}}
                          \abs{\Pi_{\frA_{m+1/2}}(u-\Pi_{2^{-m-1}\frA}u) 
                          } \dif x
                          + c\,  \dashint\limits_{\frA_{m+1/2}}
                          \abs{\Pi_{\frA_{m+1/2}}(u-\Pi_{2^{-m}\frA}u) 
                          } \dif x
    \\
                        & \leq c\,  \dashint\limits_{\frA_{m+1/2}} \abs{u-\Pi_{2^{-m-1}\frA}u} \dif x
                          +  c\,\dashint_{\frA_{m+1/2}} \abs{u-\Pi_{2^{-m}\frA}u} \dif x
    \\
                        & \leq c\,  \dashint\limits_{2^{-m-1}\frA} \abs{u-\Pi_{2^{-m-1}\frA}u}                  \dif x
                          +  c\,\dashint_{2^{-m}\frA} \abs{u-\Pi_{2^{-m}\frA}u} \dif x.
  \end{align*}
  On the other hand, since $\Pi_{\frA}$ is a projection, we have 
  \begin{align*}
    \dashint_{B}
    \abs{p_0(x) - p_0(x_0)}\dif x
    &= \dashint_{B} \abs{(\Pi_\frA u)(x) - (\Pi_\frA u)(x_0)}\dif x
    \\
    &\leq \dashint_{B} \bigabs{ \big(\Pi_\frA (u -
      \mean{u}_{\frA}\big)(x)}\dif x + \bigabs{\big(\Pi_\frA (u -
      \mean{u}_\frA)\big)(x_0)}
    \\
    &\leq c\, \dashint_{\frA} \bigabs{ \Pi_\frA (u -
      \mean{u}_{\frA})}\dif x
    \\
    &\leq c\, \dashint_{\frA} \bigabs{ u -
      \mean{u}_{\frA}}\dif x,
  \end{align*}
  where we have again used inverse estimates in the penultimate step
  and the {$\lebe^1$-stability} of~$\Pi_\frA$
  (cf.~\eqref{eq:L1stability})  in the last step.

  We collect all estimates and get
  \begin{align*}
    \dashint_{2^{-j}B} \abs{u - \mean{u}_{2^{-j}B}}\dif x
    &\leq 2^{-j} c\,\dashint_{\frA} \abs{u -
      \mean{u}_{\frA}}\dif x
    \\
    &\quad + c
      \dashint_{2^{-j}B} \abs{u -\Pi_{2^{-j} B} u}\dif x
      \\
    &\quad + c \sum_{m=0}^{j}   
      2^{m-j}\dashint_{2^{-m}\frA} \abs{u -\Pi_{2^{-m}\frA} u}\dif x.
  \end{align*}
  This proves \eqref{eq:osc-sum-main}. The proof
  of~\eqref{eq:osc-sum-main} is analogous. Starting directly
  with~$2^{-j}\frA$ instead of $2^{-j} B$ allows to avoid~\textrm{II}
  and we obtain the better right-hand side.
\end{proof}
We now derive two useful consequences of Proposition~\ref{pro:osc-sum}.
\begin{corollary}
  \label{cor:osc-sum}
 Let $k < n$. Then there exists $c=c(\A)>0$ such that for all
  $u\in\lebe^{1}(\ball;V)$ there holds for each ball~$B$ with
  center~$x_0$
  \begin{align*}
      \sum_{j=0}^{\infty} \dashint_{2^{-j}B} \abs{u - \mean{u}_{2^{-j}B}}\dif x
      &\leq c\,\dashint_{B} \abs{u -
        \mean{u}_{B}}\dif x + c\, \mathcal{I}_{k}(\A u\mres\ball)(x_0).
  \end{align*}
\end{corollary}
\begin{proof}
  Let $\frA = B \setminus \frac 14 \overline{B}$.
  We use Proposition~\ref{pro:osc-sum} and sum over~$j \in \setN_0$ to get
  \begin{align*}
    \lefteqn{\sum_{j=0}^{\infty} \dashint_{2^{-j}B} \abs{u -
    \mean{u}_{2^{-j}B}}\dif x} \qquad
    &
    \\
    &\leq c\,\dashint_{B} \abs{u -
      \mean{u}_{B}}\dif x  +\sum_{j=0}^{\infty}
                               \dashint_{2^{-j}B} \abs{u -\Pi_{2^{-j}B} u}\dif x
    \\
    &\quad + c \sum_{j=0}^{\infty}\sum_{m=0}^{j} 
                               2^{m-j}\dashint_{2^{-m}\frA} \abs{u
      -\Pi_{2^{-m}\frA} u}\dif x =:\mathrm{I}+\mathrm{II}+\mathrm{III}.
  \end{align*}
  Term $\mathrm{I}$ is already in a convenient form, whereas
  $\mathrm{III}$ can be estimated via the Cauchy product by
  \begin{align*}
    \mathrm{III} & \leq
                   c\sum_{m=0}^{\infty}\dashint_{2^{-m}\frA}|u-\Pi_{2^{-m}\mathfrak{A}}u|
                   \dif  x
    \\
                 & \leq c\sum_{m=0}^{\infty}(2^{-m}r)^{k-n}|\A
                   u|(2^{-m}\mathfrak{A})
    \\
                 & \leq c\int_{\ball}\frac{\dif |\A
                   u|(x)}{|x-x_{0}|^{n-k}} = c\, \mathcal{I}_{k}(\A
                   u\mres\ball)(x_{0}).  
  \end{align*}
  since $|x-x_{0}|\eqsim 2^{-m}r$ for any
  $x\in 2^{-m}\mathfrak{A}$ and the annuli $2^{-m}\mathfrak{A}$ have
  finite mutual overlap and all are contained in $\ball$. Turning to
  $\mathrm{II}$, we have
  \begin{align*}
    \sum_{j = 0}^{\infty} \dashint_{2^{-j} B} \abs{u - \Pi_{2^{-j}B} u}\dif x
    &\leq c
      \sum_{j = 0}^{\infty} (2^{-j}r)^{k-n}\abs{\opA u}(2^{-j}B)
    \\
    &\leq c\int_{B} \sum_{j = 0}^{\infty} (2^{-j}r)^{k-n}
      \indicator_{2^{-j} B}\dif|\A u|
    \\
    &\leq c \int_{B} \abs{x-x_0}^{k-n}\dif |\A u|
    \\
    &= \mathcal{I}_k(\opA u\mres\ball)(x_0).  
  \end{align*}
Gathering estimates, the proof is complete.
\end{proof}

\subsection{Proof of Theorem~\ref{thm:main}}
\label{sec:proofofthm}
After the preparations from the preceding subsection, we now turn to the proof of Theorem~\ref{thm:main}. 
\begin{proof}[Proof of Theorem~\ref{thm:main}~\ref{item:main1}]
  Let $u\in\bv_{\locc}^{\A}(\R^{n})$, $x_{0}\in\R^{n}$ and $r>0$ be
  such that $\mathcal{I}_{k}(\A
  u\mres\ball(x_{0},r))(x_{0})<\infty$. Then, by
  Corollary~\ref{cor:osc-sum}, we have
  \begin{align}\label{eq:Cauchy}
    \begin{split}
      |\mean{u}_{2^{-j}B}-\mean{u}_{2^{-l}B}| & \leq \sum_{i=l}^{j}|\mean{u}_{2^{-i}B}-\mean{u}_{2^{-i-1}B}| \\ & \leq 2^{n}\sum_{i=l}^{j}\dashint_{2^{-i}B}|u-\langle u\rangle_{2^{-i}B}|\dif x \to 0 
    \end{split}
  \end{align}
  as $j,l\to\infty$. Therefore, $\mean{u}_{2^{-i}B}\to q$ as
  $i\to\infty$ for some $q\in V$ and so, by
  Corollary~\ref{cor:osc-sum},
  \begin{align*}
    \dashint_{2^{-i}B}|u-q|\dif x \leq \dashint_{2^{-i}B}|u-\mean{u}_{2^{-i}B}|\dif x + |\mean{u}_{2^{-i}B}-q| \to 0.
  \end{align*}
  Since it suffices to consider balls $2^{-i}B$, the proof is complete. 
\end{proof}
\begin{proof}[Proof of Theorem~\ref{thm:main}~\ref{item:main2}]
  Recall that now $k=n$. Let $x_0 \in \setR^n$; our first objective is
  to prove that $x_{0}$ is a Lebesgue point for
  $u\in\bv^{\A}(\R^{n})$.  Let $B = B(x_0,r)$ and
  $\frA := B \setminus \frac 14 \overline{B}$. Then it follows from
  Proposition~\ref{pro:osc-sum} and
  Corollary~\ref{cor:poincare-annulus} (applied with~$\lambda=0$ and
  $\lambda=\frac 14$) that
  \begin{align*}
    \dashint_{2^{-j} B} \abs{u - \mean{u}_{2^{-j}B}}\dif x
    &\leq
      2^{-j} c\,\dashint_{\frA} \abs{u - \mean{u}_{\frA}}\dif x
    \\
    &\quad
      + c\, \abs{\opA u}(2^{-j} B \setminus \set{x_0})
      + c \sum_{m=0}^{j} 2^{m-j} \abs{\opA u}(2^{-m}\frA)
    \\
    &\leq
      2^{-j} c\,\dashint_{B} \abs{u - \mean{u}_{B}}\dif x + c
      \abs{\opA u}(B \setminus \set{x_0}). 
  \end{align*}
  Thus, if follows that for $0<s <\frac{r}{2}$ we have
  \begin{align*}
    \dashint_{B(x_0,s)} \abs{u - \mean{u}_{B(x_0,s)}}\dif x
    &\leq
      c\,\frac{s}{r} \dashint_{B(x_0,r)} \abs{u - \mean{u}_{\frA}}\dif x + c
      \abs{\opA u}(B(x_0,r) \setminus \set{x_0}). 
  \end{align*}
  Let $\epsilon>0$ be arbitrary. Since
  $B(x_0,r) \setminus \set{x_0} \to \emptyset$ for $r \to 0$
  and~$\abs{\A u}$ is a Radon measure, hence outer regular, we find~$r>0$ such that $
  \abs{\opA u}(B(x_0,r) \setminus \set{x_0}) < \frac \epsilon 2$. Now,
  we can choose~$s$ so small that also the integral on the right hand
  side becomes small than~$\frac{\epsilon}{2}$. Since $\epsilon>0$ was
  arbitrary this proves that
  \begin{align}
    \label{eq:vmo}
    \lim_{s \searrow 0} \dashint_{B(x_0,s)} \abs{u -
    \mean{u}_{B(x_0,s)}}\dif x &= 0.
  \end{align}
  Recall that $B=B(x_0,r)$ and put
  $\frA := B \setminus \frac 14 \overline{B}$. Then~\eqref{eq:vmo}
  implies that
  \begin{align}
    \label{eq:diff-A-B-mean}
    \lim_{j \to \infty} \abs{\mean{u}_{2^{-j} B} - \mean{u}_{2^{-j}\frA}}
    &\leq c\,\lim_{j \to \infty} \dashint_{2^{-j} B} \abs{u -
      \mean{u}_{2^{-j} B}}\dif x = 0.
  \end{align}
  Now, by Proposition~\ref{pro:osc-sum} and
  Corollary~\ref{cor:poincare-annulus} (with~$\lambda=0$ and $\lambda=
  \frac 14$) and $k=n$ we have
  \begin{align*}
    \lefteqn{\abs{\mean{u}_{2^{-j} \frA} -
    \mean{u}_{2^{-j-1}\frA}} } \qquad
    &
      \\
    &\leq
      c\,\dashint_{2^{-j}\frA} \abs{u - \mean{u}_{2^{-j}\frA}}\dif x
    \\
    &\leq
      c\,2^{-j} \dashint_{\frA} \abs{u - \mean{u}_{\frA}}\dif x
      + c \sum_{m=0}^{j} 2^{m-j}\dashint_{2^{-m}\frA} \abs{u
      -\Pi_{2^{-m}\frA} u}\dif x
    \\
    &\leq
      c\, 2^{-j} \dashint_{\frA} \abs{u - \mean{u}_{\frA}}\dif 
      + c \sum_{m=0}^{j} 2^{m-j} \abs{\A u}(2^{-m}\frA).
  \end{align*}
  Summing over~$j \geq 0$ we obtain
  \begin{align}
    \label{eq:meanA-summable}
    \begin{aligned}
      \sum_{j \geq 0} \abs{\mean{u}_{2^{-j} \frA} -
        \mean{u}_{2^{-j-1}\frA}} &\leq c\,\dashint_{\frA}
      \abs{u - \mean{u}_{\frA}}\dif x + c \sum_{j\geq 0}
      \sum_{m=0}^{j} 2^{m-j} \abs{\A u}(2^{-m}\frA)
      \\
      &\leq c\,\dashint_{\frA} \abs{u - \mean{u}_{\frA}}\dif x +
      c\, \abs{\A u}(B \setminus \set{x_0}).
    \end{aligned}
  \end{align}
  This proves that $u_0 := \lim_{j \to \infty} \mean{u}_{2^{-j} \frA}$
  exists. Due to~\eqref{eq:diff-A-B-mean} we see that also
  $u_0 = \lim_{j \to \infty} \mean{u}_{2^{-j} B}$. Thus, it follows
  by~\eqref{eq:vmo} that
  \begin{align*}
    \lim_{j \to \infty} \dashint\limits_{2^{-j}B} \abs{u - u_0}\dif x
      &\leq
    \limsup_{j \to \infty} \dashint\limits_{2^{-j}B} \abs{u -
        \mean{u}_{2^{-j} B}}\dif x + \limsup_{j \to \infty}
        \abs{\mean{u}_{2^{-j} B}-u_0} = 0.
  \end{align*}
  This proves that~$x_0$ is a Lebesgue point. Since $x_0$ was
  arbitrary, we see that all points of~$u$ are Lebesgue points. In the
  following we choose~$u$ to be the unique representative, which
  coincides with Lebesgue point limits.

  Again let $B$ be ball with center~$x_0$ with radius~$r>0$. Now, let
  $y \in \frac 12 B$ be fixed. Then we can choose a ball
  $B' \subset B \setminus \frac 12 \overline{B}$ such that the sets
  \begin{align*}
    C_{x_0} &:= \convexhull(\overline{B'} \cup \set{x_{0}}),
    \\
    C_y &:= \convexhull(\overline{B'} \cup \set{y}),
  \end{align*}
  satisfy $y \notin C_{x_0}$ and $x \notin C_y$. This geometric constellation is depicted in Figure~\ref{fig:continuity}. 
 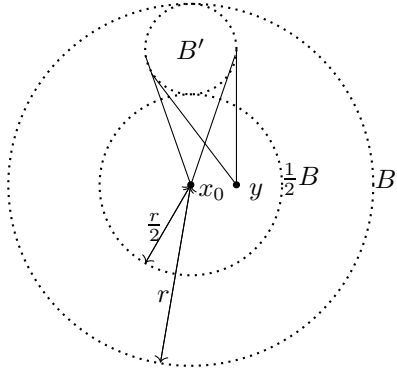
\begin{figure}
    \begin{tikzpicture}[scale=0.8]
      \draw [dotted, thick] (-1,0) circle [radius = 3cm];
      \draw [dotted, thick] (-1,0) circle [radius = 1.5cm];
      \draw [dotted, thick] (-1,2.25) circle [radius = 0.75cm];
      \draw [fill=black] (-1,0) circle [radius=0.05];
      \draw [fill=black] (-0.25,0) circle [radius=0.05];
      \node at (-0.65,-0.125) {$x_{0}$}; 
      \node at (0.075,-0.125) {$y$};
      \node at (-1,2.25) {$\ball'$};
      \draw (-1,0) to (-1.75,2.15);
      \draw (-1,0) to (-0.27,2.15);
      \draw (-0.25,0) to (-0.25,2.25);
      \draw (-0.25,0) to (-1.65,1.85);
      \draw [fill=black, <->] (-1,0) to (-1.5,-2.95);
      \node at (-1.45,-1.85) {$r$};
      \draw [fill=black, <->] (-1,0) to (-1.75,-1.3);
      \node at (-1.6,-0.7) {$\tfrac{r}{2}$};
      \node at (0.8,0.1) {$\frac 12 B$};
      \node at (2.2,0.1) {$B$};
    \end{tikzpicture}
    \caption{The geometric situation in the proof of Theorem~\ref{thm:main}\ref{item:main2}.}
    \label{fig:continuity}
 \end{figure}
  Now, we use the representation formula~\eqref{eq:repformula} with $B_\Omega = B'$ to get
  \begin{align}
    \label{eq:2}
    \begin{aligned}
      \abs{u(x_0) - (\mathbb{P}^m_{B'} u)(x_0)} &\leq c\,\abs{\bbA
        u}(C_{x_0}),
      \\
      \abs{u(y) - (\mathbb{P}^m_{B'} u)(y)} &\leq c\,\abs{\bbA
        u}(C_y).
    \end{aligned}
  \end{align}
  This can be improved to
  \begin{align}\label{eq:improve}
  \begin{split}
    \abs{u(x_0) - (\mathbb{P}^m_{B'} u)(x_0)} &\leq c\,\abs{\bbA u}(C_{x_0}
                                            \setminus \set{x_{0}}),
    \\
    \abs{u(y) - (\mathbb{P}^m_{B'} u)(y)} &\leq c\,\abs{\bbA u}(C_y
                                            \setminus \set{y}). 
                                            \end{split}
  \end{align}
  as follows: It suffices to prove the first estimate. Denote $x_0'$
  the center of~$B'$ and put
  $x_j := (1-\frac 1j)x_{0} + \frac 1j x_0'$ and
  choose~$\theta \in (0,1)$ so small that
  $B(x_j, 2\theta \abs{x_j-x_0}) \subset C_{x_0} \setminus
  \set{x_0}$. Taking the average of the representation
  formula~\eqref{eq:repformula} for
  every~$x \in B(x_j, 2\theta \abs{x_j-x_0})$ we obtain
  \begin{align*}
     \bigabs{ \mean{u}_{B(x_j, 2\theta
     \abs{x_j-x_0})} - \mean{\mathbb{P}_{\ball'}^{m}u}_{B(x_j, 2\theta
    \abs{x_j-x_0})}} &\leq c\, \abs{\opA u}(\convexhull(\overline{B'}
    \cup \set{x_0}))
    \\
    &\leq c|\A u|(C_{x_{0}}\setminus\{x_{0}\})
  \end{align*}
  using also that
  $\convexhull(\overline{B'} \cup \set{x_0}) \subset
  C_{x_{0}}\setminus\{x_{0}\}$ for every
  $x \in B(x_j, 2\theta \abs{x_j-x_0})$.  Now, \eqref{eq:improve}
  follows by passing with~$j \to \infty$ and using that~$x_0$ is a
  Lebesgue point and that $\mathbb{P}_{\ball'}^{m}u$ is continuous.
  
  Using~\eqref{eq:improve} we get
  \begin{align*}
    \abs{u(x_0)-u(y)}
    &\leq
      \abs{u(x_0) - (\mathbb{P}^m_{B'} u)(x_0)} 
      +\abs{u(y) - (\mathbb{P}^m_{B'} u)(x_0)} \\ &
      + \abs{(\mathbb{P}^m_{B'}u)(x_0)) - (\mathbb{P}^m_{B'} u)(y)}
    \\
    &\leq c\,\abs{\bbA u}(C_{x_0}
      \setminus \set{x_0}) + c\,\abs{\bbA u}(C_y
      \setminus \set{y}) + \abs{(\mathbb{P}^m_{B'} u)(x_0)) - (\mathbb{P}^m_{B'} u)(y)}
    \\
    &\leq c\,\abs{\bbA u}(B \setminus \set{x_0})
      + \abs{(\mathbb{P}^m_{B'} u)(y)) - (\mathbb{P}^m_{B'} u)(x_0)}.
  \end{align*}
  We further estimate 
  \begin{align*}
    \abs{(\mathbb{P}^m_{B'} u)(y)) - (\mathbb{P}^m_{B'} u)(x_0)}
    &\leq \abs{x_0-y} \,\norm{\nabla (\mathbb{P}^m_{B'} u)}_{L^\infty(B)}
    \\
    &\leq c\,\frac{\abs{x_0-y}}{r} \dashint_{B'} \abs{\mathbb{P}^m_{B'} (u-\mean{u}_B)}\dif x
    \\
    &\leq c\,\frac{\abs{x_0-y}}{r} \dashint_{B'} \abs{u-\mean{u}_B}\dif x
    \\
    &\leq c\,\frac{\abs{x_0-y}}{r} \dashint_{B} \abs{u-\mean{u}_B}\dif x,
  \end{align*}
  where we have used inverse estimates for polynomials and the
  $L^1$-stability of the averaged Taylor polynomial. Overall, we obtain
  \begin{align}
    \label{eq:continuity}
    \abs{u(x_0)-u(y)}
    &\leq  c\,\abs{\bbA u}(B \setminus \set{x_0}) 
    + c\,\frac{\abs{x_0-y}}{r} \dashint_{B} \abs{u-\mean{u}_B}\,dx.
  \end{align}
  This proves that~$u$ is continuous at~$x_0$ and
  also~\eqref{eq:cont-est-thm}. Indeed, for given $\varepsilon>0$ we
  choose $r>0$ such that
  $|\A u|(\ball\setminus\{x_{0}\})<\varepsilon$. Then, taking the
  constant $c>0$ from \eqref{eq:continuity}, we choose $0<\delta<r$ so
  small such that $|x_{0}-y|<\delta$ implies
  $c|x_{0}-y| r^{-1}\dashint_{\ball}|u|\dif z < \varepsilon$. This
  yields $|u(x_{0})-u(y)|<c\varepsilon$ for all $y\in\R^{n}$ with
  $|x_{0}-y|<\delta$, and so $u$ is continuous. The proof is complete.
\end{proof}
\begin{proof}[Proof of Theorem~\ref{thm:main}~\ref{item:main3}]
  Recall that now $\A$ is a $\mathbb{C}$-elliptic differential operator of
  order $k>n$. Then we can proceed exactly as in the
  case~\ref{item:main2} but with $\nabla^{k-n} u$ instead of~$u$. In
  the applications of the \Poincare{} type estimates on the annuli
  (see Corollary~\ref{cor:poincare-annulus}) we choose~$\ell =
  k-n$. With the same arguments as in the proof of
  Theorem~\ref{thm:main}\ref{item:main2} we find that every point 
  of~$\nabla^{n-k} u$ is a Lebesgue point and that for all~$x \in \Rn$
  and $r>0$
  \begin{align}
    \label{eq:continuity-higher}
    \begin{aligned}
      \abs{(\nabla^{k-n} u)(x_0)-(\nabla^{k-n} u)(y)} &\leq
      c\,\abs{\bbA u}(B \setminus \set{x_0})
      \\
      &\quad + c\,\frac{\abs{x_0-y}}{r} \dashint_{B}
      \abs{\nabla^{k-n} u -\mean{\nabla^{k-n} u}_B}\dif x.
    \end{aligned}
  \end{align}
  for all $y\in\R^{n}$ with $|x-y|<\tfrac{1}{2}r$. This implies again that~$\nabla^{k-n} u$ is continuous, so $u \in
  C^{k-n}$. This proves the claim.
\end{proof}
\begin{remark}
  \label{rem:Linfty}
  Note that from the representation formula, see also~\eqref{eq:2}, we
  immediately obtain for~$k\geq n$ the local $L^\infty$-bound
  \begin{align}
    \norm{\nabla^{k-n} u}_{L^\infty(B)} &\leq c \dashint_B
                                          \abs{\nabla^{k-n} u}\,dx +
                             c\,\abs{\opA u}(B). 
  \end{align}
  This also implies (with $B \to \Rn$) that for~$k \geq n$
  \begin{align}
    \norm{\nabla^{k-n} u}_{L^\infty(\Rn)} &\leq
                             c\,\abs{\opA u}(\Rn). 
  \end{align}
  Moreover, the extension operator from \cite{GR1} allows us to apply
  these estimates and Theorem~\ref{thm:main} to bounded Lipschitz
  domains. In particular, we obtain for $k\geq n$ the embedding
  $\bv^{\A}(\Omega)\hookrightarrow\hold^{k-n}(\overline{\Omega};V)$.
\end{remark}

\begin{remark}[Singletons]
  The step from~\eqref{eq:2} to~\eqref{eq:improve} can also be
  obtained by a different argument, namely that $\abs{\opA
    u}(\set{x_0})=0$ for each~$x_0$, i.e. $\opA u$ cannot charge singletons. 

  One possibility to prove this is that every~$\setC$-elliptic operator 
  is elliptic and cancelling (cf.~\cite{GR1}) and thus by~\textsc{Van Schaftingen}
  \cite[Prop.~2.1]{VS13}, $\abs{\opA u}$ cannot charge singletons.

  Let us present an alternative proof based on the trace theorem
  of~\cite{BDG}, for simplicity stated for first order operators. Let $\Sigma$ be a Lipschitz hypersurfaces passing
  through the points~$x_0$. Then by the gluing
  theorem~\cite[Proposition~4.12]{BDG}
  \begin{align}\label{eq:jumps}
    \A u\mres \Sigma =
    (u_{\Sigma}^{+}-u_{\Sigma}^{-})\otimes_{\A}
    \nu\mathscr{H}^{n-1}\mres \Sigma,
  \end{align}
  where $\nu$ is the unit normal to~$\Sigma$ and
  $u_{\Sigma}^{+},u_{\Sigma}^{-}$ are the left- or right-sided traces
  along $\Sigma$, respectively, which exist in
  $\lebe^{1}(\Sigma;V)$
  by~\cite[Thm.~1.1]{BDG}. This implies
  \begin{align}\label{eq:AC}
    \A u\mres\Sigma \ll \mathscr{H}^{n-1}\mres\Sigma.
  \end{align}
  Thus $\opA u$ cannot charge any $\mathscr{H}^{n-1}$-nullset contained in~$\Sigma$. Thus, in particular, $\abs{\opA u}(\set{x_0}) =0$. Let us note that, based on the proof of \cite[Thm.~4.18]{BDG}, for elliptic first order operators $\A$, \eqref{eq:AC} for all $u\in\bv^{\A}(\R^{n})$ and Lipschitz hypersurfaces $\Sigma$ is in fact equivalent to $\mathbb{C}$-ellipticity. 
\end{remark}
\begin{remark}  
The representability \eqref{eq:jumps} could also be used to describe points in the jump set $J_{u}$ on hypersurfaces. However, working from Theorem~\ref{thm:main}\ref{item:main1}, it is not immediately clear how the Riesz potential criterion should yield $|\A u|(S_{u}\setminus J_{u})=0$, which would be required for a proper structure theory for $\bv^{\A}$-maps. We intend to tackle this question in the future. 
\end{remark}


\begin{thebibliography}{99}
\bibitem{AdamsHedberg} Adams, D.R.; Hedberg, L.I.: Function spaces and potential theory. Grundlehren der mathematischen Wissenschaften 314, 1996. 
\bibitem{AFP} Ambrosio, L.; Fusco, N.; Pallara, D.: Functions of bounded variation and free discontinuity problems. Oxford University Press, 2000. 
\bibitem{VS14} Bousquet, P.; Van Schaftingen, J.: Hardy–Sobolev inequalities for vector fields and canceling linear differential operators, Indiana Univ. Math. J. 63 (2014), no. 5, 1419--1445.
\bibitem{BD} Breit, D.; Diening, L.: Sharp conditions for Korn inequalities in Orlicz spaces. J. Math. Fluid Mech.
14 (2012), no. 3, 565–573.
\bibitem{BDG} Breit, D.; Diening, L.; Gmeineder, F.: On the trace operator for functions of bounded $\A$-variation. To appear at Anal. PDE. 
\bibitem{CalderonZygmund} Calder\'{o}n, A.; Zygmund, A.: On the existence of certain singular integrals. Acta. Math. 88 (1952) pp.~85--139.
\bibitem{DRS} Diening, L.; Ruzicka, M.; Schumacher, K.: A Decomposition technique for John domains, Annales Academiae Scientiarum Fennicae (2009), vol. 35, 87-114.
%\bibitem{DG} Diening, L.; Gmeineder, F.: Limiting $\lebe^{1}$-trace inequalities for differential operators. In preparation. 
\bibitem{Korn} Friedrichs, K.: On the boundary value problems of the theory of elasticity and Korn’s inequality. Ann. Math. 48(2),
441–471 (1947).
\bibitem{FuchsSeregin} Fuchs, M.; Seregin, G.: Variational methods for problems from plasticity theory and for generalized
Newtonian fluids. Lecture Notes in Mathematics, 1749. Springer-Verlag, Berlin, 2000. vi+269 pp.
\bibitem{GR1} Gmeineder, F.; Raita, B.: Embeddings for $\A$-weakly differentiable functions on domains. J. Func. Anal., Vol. 277 (12), 2019.
\bibitem{GR2} Gmeineder, F.; Raita, B.: Limiting $\lebe^{p}$-differentiability of $\bd$-maps. Rev. Mat. Iberoam., Vol. 35 (7), 2019, pp. 2071--2078. 
\bibitem{GR3} Gmeineder, F.; Raita, B.; Van Schaftingen, J.: Limiting Trace Inequalities for Vectorial Differential Operators. ArXiv preprint: arXiv:1903.08633.
\bibitem{Hoermander}  H\"{o}rmander, L.: Differentiability properties of solutions of systems of differential equations.
Ark. Mat. 3, 527--535 (1958).
\bibitem{Kalamajska} Ka\l{}amajska, A.: Pointwise multiplicative inequalities and Nirenberg type estimates in weighted Sobolev spaces, Studia Math. 108 (1994), no. 3, 275--290.
\bibitem{KK0} Kirchheim, B.; Kristensen, J.:  Automatic convexity of rank–1 convex functions, C. R. Math. Acad. Sci. Paris 349 (2011), no. 7--8, 407--409. 
\bibitem{KK}  Kirchheim, B.; Kristensen, J.: On rank one convex functions that are homogeneous of degree one. Arch. Ration. Mech. Anal. 221 (2016), no. 1, 527-558.
\bibitem{Mosolov} Mosolov, P.P., Mjasnikov, V.P.: On the correctness of boundary value problems in the mechanics of continuous
media. Math. USSR Sbornik 17(2), 257–267 (1972)
\bibitem{Ornstein} Ornstein, D.: A non--equality for differential operators in the $\lebe_{1}$-norm, Arch. Rational Mech. Anal. 11 (1962), 40-49.
\bibitem{Raita1} Raita, B.: Critical $\lebe^{p}$-differentiability of $\bv^{\A}$-maps and canceling operators. Trans. Amer. Math. Soc. 372 (2019), 7297-7326. 
\bibitem{Raita2} Raita, B.; Skorobogatova, A.: Continuity and canceling operators of order $n$ on $\R^{n}$. ArXiv preprint: arXiv:1903.03574
\bibitem{Smith} Smith, K.T.: Formulas to represent functions by their derivatives. Math. Ann. 188 (1970), 53--77.
\bibitem{Spencer}  Spencer, D. C.: Overdetermined systems of linear partial differential equations. Bull. Amer. Math. Soc. 75, 179--239 (1969).
\bibitem{Tri92} Triebel, H.: Theory of Function Spaces II. Birkh\"{a}user Modern Classics (2010), reprint of the 1992 edition. 
\bibitem{VS13} Van Schaftingen, J.: Limiting Sobolev inequalities for vector fields and cancelling linear differential operators. Journal of the European Mathematical Society, 2013, 15(3), 877-921.
\bibitem{VS14a} Van Schaftingen, J.: Limiting Bourgain–Brezis estimates for systems of linear differential equations: Theme and variations, J. Fixed Point Theory Appl. 15 (2014), no. 2, 273--297.
\end{thebibliography}
\end{document}